\documentclass[12pt]{extarticle}
\usepackage{amsmath, amsthm, amssymb, hyperref, color}
\usepackage{graphicx}
\usepackage[all]{xypic}
\usepackage{verbatim}
\usepackage{tikz}

\oddsidemargin \evensidemargin
\usepackage{color}
\newif\ifincludeprevious

\newtheorem{theorem}{Theorem}
\numberwithin{theorem}{section}
\newtheorem{proposition}[theorem]{Proposition}
\newtheorem{lemma}[theorem]{Lemma}

\newtheorem{remark}[theorem]{Remark}
\newtheorem{example}[theorem]{Example}

\newcommand{\lra}{\longrightarrow}
\renewcommand{\epsilon}{\varepsilon}

\newcommand{\tnull}{\text{null}}

\usepackage{mathtools}
\usepackage{microtype}

\numberwithin{equation}{section}

\date{}
\title{\textbf{On the Schottky problem for genus five Jacobians with a vanishing theta null}}

\author{Daniele Agostini and Lynn Chua}

\begin{document}

\maketitle

\begin{abstract}
We give a solution to the weak Schottky problem for genus five
Jacobians with a vanishing theta null, answering a question of
Grushevsky and Salvati Manni. More precisely, we show that if a
principally polarized abelian variety of dimension five has a
vanishing theta null with a quadric tangent cone of rank at most three,
then it is in the Jacobian locus, up to extra irreducible
components. We employ a degeneration argument, together with a study
of the ramification loci for the Gauss map of a theta divisor.
\end{abstract}

\section{Introduction} \label{section:intro}

The \emph{Schottky problem} asks to recognize the Jacobian locus
$\mathcal{J}_g$ inside the moduli space $\mathcal{A}_g$ of principally polarized abelian varieties of dimension at least $4$, and
it has been a central problem in algebraic geometry for more than a
century. In the most classical interpretation, the problem asks for
equations in the period matrix $\tau$ that vanish exactly on the
Jacobian locus. In this form, the Schottky problem is completely solved only in genus $4$, with an explicit equation given by Schottky \cite{Schottky1888} and Igusa \cite{IgusaSchottkyFour1981}. The \emph{weak Schottky problem} asks for explicit equations that characterize Jacobians up to extra irreducible components. A solution to this problem was given in genus $5$ by Accola \cite{AccolaSchottkyFive1983} and, in a recent breakthrough, by Farkas, Grushevsky and Salvati Manni in all genera \cite{FarkasGrushevskySalvatiManniWeakSchottky2017}. 

There have been many other approaches to the Schottky problem over the
years; see \cite{GrushevskySchottkySurvey2012} for a recent
survey. One noteworthy approach is the study of the singularities of
the theta divisor. Indeed, Riemann's singularity theorem says that the theta divisor $\Theta$ of a genus $g$ Jacobian has a singular locus of dimension at least $g-4$. Andreotti and Mayer \cite{AndreottiMayerSingular1967} proved that the locus of abelian varieties with this property contains $\mathcal{J}_g$ as an irreducible component, thus providing a solution to the weak Schottky problem. Moreover, the singularities of the theta divisor have a deep relationship with the geometry of $\mathcal{A}_g$. In genus five, this has been thoroughly explored in \cite{FarkasGrushevskySalvatiManniVerraSingularities2014}.

In this work, we focus on Jacobians with a vanishing theta null, which
is an even two-torsion point in the theta divisor. The abelian
varieties with this property have been intensely studied
\cite{MumfordSiegel1983, BeauvillePrym1977, DebarreLieu1992,
  GrushevskySalvatiManniThetaTwo2007, GrushevskySalvatiManniVanishing2008} and they form a divisor $\theta_{\tnull}$ in $\mathcal{A}_g$. The Schottky problem in this case becomes that of recognizing $\mathcal{J}_g\cap \theta_{\tnull}$ inside $\theta_{\tnull}$.  

The first observation is that a vanishing theta null is automatically a singular point
of the theta divisor. Hence, following the Andreotti-Mayer philosophy,
one is led to study the local structure of $\Theta$ around the
singular point, and the first natural invariant is the \emph{rank of the quadric tangent cone}. More precisely, if the divisor is cut out by a theta function, $\Theta = \{ \theta=0 \}$, then the quadric tangent cone $Q_p\Theta$ at $p\in \Theta^{sing}$ is defined by the Hessian matrix evaluated at $p$:
\begin{equation}
Q_p\Theta \sim \begin{pmatrix} 
\frac{\partial^2\theta}{\partial {z_1}^2} & \frac{\partial^2\theta}{\partial z_1\partial z_2} & \cdots & \frac{\partial^2 \theta}{\partial z_1 \partial z_{g}}  \\   
\frac{\partial^2\theta}{\partial z_1\partial z_2} & \frac{\partial^2\theta}{\partial {z_2}^2} & \cdots & \frac{\partial^2 \theta}{\partial z_2 \partial z_{g}}  \\ 
\vdots & \vdots & \ddots & \vdots &  \\
\frac{\partial^2\theta}{\partial z_1\partial z_{g}} & \frac{\partial^2\theta}{\partial {z_2}\partial{z_{g}}} & \cdots & \frac{\partial^2 \theta}{\partial {z}^2_{g}}
\end{pmatrix} \,.
\end{equation}   
The rank of $Q_p\Theta$ is the rank of the Hessian. This leads to a stratification of $\theta_{\tnull}$, first introduced by Grushevsky and Salvati Manni \cite{GrushevskySalvatiManniVanishing2008},
\begin{equation}
\theta_{\text{null}}^0 \subseteq \theta_{\text{null}}^1 \subseteq \cdots\ \subseteq \theta^{g-1}_{\tnull} \subseteq \theta^{g}_{\tnull} = \theta_{\tnull}
\end{equation}
where $\theta^h_{\tnull}$ is the locus of abelian varieties with a
vanishing theta null, with a quadric tangent cone of rank at most $h$. In particular, if a Jacobian has a vanishing theta null, then a result of Kempf \cite{Kempf1973} shows that the quadric tangent cone has rank at most three, hence
\begin{equation}
\mathcal{J}_g \cap \theta_{\tnull} \subseteq \theta^3_{\tnull}\,.
\end{equation}
Grushevsky and Salvati Manni proved in
\cite{GrushevskySalvatiManniVanishing2008} that this inclusion is
actually an equality in genus $4$, confirming a conjecture of
H. Farkas. In the same paper, they ask whether $\mathcal{J}_g\cap
\theta_{\tnull}$ is an irreducible component of $\theta^3_{\tnull}$ in
higher genera, which would imply a solution to the weak Schottky problem for Jacobians with a vanishing theta null. Our main result is an affirmative answer to this question in genus $5$.

\begin{theorem}\label{thm:mainthm}
In genus five, the locus $\mathcal{J}_5\cap \theta_{\tnull}$ is an irreducible component of $\theta^3_{\tnull}$.
\end{theorem} 

This solution to the weak Schottky problem is very much in the
classical spirit of finding explicit equations in the period
matrix. Indeed, the condition of having an even two-torsion point in
the theta divisor can be checked by evaluating the theta function
at these (finitely many) points, and then the rank of the
Hessian can be computed numerically. We include an example of
this computation in Example~\ref{example:computations}, where we use the Julia package presented in our companion paper \cite{AgostiniChuaNumerical2019}. For a similar computational approach to the Schottky problem in genus four, see \cite{ChuaKummerSturmfelsSchottky2018}.

The strategy of our proof is to bound the dimension of
$\theta^3_{\tnull}$, by working over a partial compactification
$\overline{\mathcal{A}_g}^1 = \mathcal{A}_g\cup \partial
\mathcal{A}_g$ of $\mathcal{A}_g$, following
\cite{GrushevskySalvatiManniThetaTwo2007}. Via a study of nodal
curves with a theta characteristic, we can describe the intersection
$(\mathcal{J}_5\cap \theta_{\tnull})\cap \partial{\mathcal{A}_g}$. We can then bound the dimension $\theta^3_{\tnull}\cap
\partial{\mathcal{A}_g}$ via a study of the ramification loci, or
Thom-Boardman loci, for the Gauss map of the theta divisor, leading
to the proof of our main theorem. In particular, the study of these
ramification loci fits with classical and recent work on the Gauss map
\cite{AndreottiTorelli1958, AdamsetalGauss1993,KraemerMonodromyGauss2015,
  CodogniGrushevskySernesiGauss2017, AuffarthCodogniSalvariManniGauss2019}, and might give a way to extend our main result in higher dimensions.
\medskip

\textbf{Acknowledgments:} We are grateful to Bernd Sturmfels for
suggesting to study the Schottky problem in genus five, and for his
continuous encouragement. We thank Paul Breiding, Gavril Farkas, Sam
Grushevsky, Thomas Kr\"amer, Riccardo Salvati Manni, Andrey Soldatenkov, Sasha Timme and Sandro Verra for
useful comments and discussions. This project was initated at the Max
Planck Institute for Mathematics in the Sciences in Leipzig, which both authors would like to thank for the hospitality and support at various stages of this work.

\section{Preliminaries}

We start by defining some standard notation and background, referring
the reader to \cite{IgusaThetafunctions1972, BirkenhakeLange2004} for more details. The \emph{Siegel
  upper-half space} $\mathbb{H}_g$ is the set of complex symmetric $g\times g$ matrices with positive definite imaginary part. A
\emph{characteristic} is an element $m \in
(\mathbb{Z}/{2}\mathbb{Z})^{2g}$, which we will represent as a vector $m = [\epsilon,\delta]$, where $\epsilon,\delta \in \{0,1\}^g$. The \emph{theta function with characteristic} $m$ is the holomorphic function $\theta[m] \colon \mathbb{C}^g \times \mathbb{H}_g \to \mathbb{C}$ defined as
\begin{equation}
\theta[m](z,\tau) := 
\sum_{n\in \mathbb{Z}^g} \exp \left( \pi i\left(n+\frac{\epsilon}{2}\right)^t\tau \left(n+\frac{\epsilon}{2}\right) + 2\pi i \left(n+\frac{\epsilon}{2}\right)^t\left(z+\frac{\delta}{2}\right) \right)\,.
\end{equation}
The theta function with
characteristic $m=[0,0]$ is called the \emph{Riemann theta function},
which we will sometimes denote simply by $\theta(z,\tau)$. All theta functions satisfy the fundamental \emph{heat equation},
\begin{equation}\label{eq:heatequation}
\frac{\partial^2 \theta[m]}{\partial{z_i}\partial{z_j}} (z,\tau) = (1+\delta_{ij}) \cdot  2\pi i \frac{\partial \theta[m]}{\partial \tau_{ij}} (z,\tau)\,.
\end{equation}
The \emph{sign} of a characteristic $m$ is defined as $e(m) =
(-1)^{\epsilon^t \delta}$, and we call a characteristic \emph{even} or
\emph{odd} if the sign is $1$ or $-1$ respectively. As a function of
$z$, $\theta[m](z,\tau)$ is even (respectively odd) if and only if the characteristic $m$ is even (respectively odd). There are $2^{g-1}(2^g+1)$ even theta characteristics and $2^{g-1}(2^g-1)$ odd theta characteristics.
The \emph{theta constants} are the functions on $\mathbb{H}_g$ obtained by evaluating the theta functions with characteristic at $z=0$,
\begin{equation}
\theta[m](\tau) := \theta[m](0,\tau)\,.
\end{equation} 
The theta constants corresponding to odd characteristics vanish identically.

Now we recall the construction of the moduli space
$\mathcal{A}_g$. For every $\tau \in \mathbb{H}_g$ one defines a
principally polarized abelian variety (ppav) by $A_{\tau} = \mathbb{C}^g / \Lambda_{\tau}$, where $\Lambda_{\tau} = \mathbb{Z}^g \oplus \tau\mathbb{Z}^g$. The polarization on $A_{\tau}$ is given by the \emph{theta divisor} $\Theta_{\tau} = \left\{  z  \,|\, \theta(z,\tau) = 0 \right\}$. Observe that the symmetric theta divisors on $A_{\tau}$ are  precisely the divisors $\Theta^{[m]}_{\tau} = \{ z \,|\, \theta[m](z,\tau) = 0 \}$, which are translates of $\Theta_{\tau}$ by two-torsion points.

The modular group $\Gamma_g =\operatorname{Sp}(2g,\mathbb{Z})$ acts on
$\mathbb{H}_g$ as follows. For $\gamma\in\Gamma_g$ and $\tau\in\mathbb{H}_g$,
\begin{equation}
\gamma = \begin{pmatrix} A & B \\ C & D \end{pmatrix}\,, \qquad \gamma\cdot \tau = (A\tau + B)(C\tau+D)^{-1}\,.
\end{equation}
Two ppavs $(A_\tau,\Theta_{\tau})$ and $(A_{\tau'},\Theta_{\tau'})$
are isomorphic if and only if $\tau'=\gamma\cdot \tau$ for some
$\gamma \in \Gamma_g$, in which case an isomorphism is given by
\begin{equation}\label{eq:explicitiso}
A_{\tau} \lra A_{\gamma\cdot \tau}\,, \qquad z\mapsto (C\tau+D)^{-t}z\,.
\end{equation}

The quotient $ \mathcal{A}_g =  \mathbb{H}_g / \operatorname{Sp}(2g,\mathbb{Z}) $ is the \emph{moduli space of principally polarized abelian varieties of dimension $g$}. It is an analytic space of dimension
\begin{equation}
\dim \mathcal{A}_g = \dim \mathbb{H}_g = \frac{g(g+1)}{2}\,.
\end{equation} 

We next describe how the theta functions transform under the modular
group $\Gamma_g$. The action on $\mathbb{H}_g$ extends to an action on $\mathbb{C}^g\times \mathbb{H}_g$  by 
\begin{equation}
\gamma = \begin{pmatrix} A & B \\ C & D \end{pmatrix}\,, \qquad \gamma\cdot (z,\tau) = ((C\tau+D)^{-t} z,\gamma\cdot \tau)\,.
\end{equation}
The corresponding action on the set of characteristics $(\mathbb{Z}/2\mathbb{Z})^{2g}$ is
\begin{equation}
\gamma = \begin{pmatrix} A & B \\ C & D \end{pmatrix}\,, \qquad \gamma\cdot \begin{bmatrix} \epsilon \\ \delta \end{bmatrix} = \begin{pmatrix} D & -C \\ -B & A \end{pmatrix}\begin{bmatrix} \epsilon \\ \delta \end{bmatrix} + \begin{bmatrix} \operatorname{diag}(CD^t) \\ \operatorname{diag}(AB^t) \end{bmatrix}\,.
\end{equation}
This can be combined with the action of the lattice $\mathbb{Z}^g$ to give the \emph{universal family} of abelian varieties $\mathcal{X}_g = (\mathbb{C}^g\times \mathbb{H}_g)/\mathbb{Z}^{2g}\rtimes \Gamma_g \to \mathcal{A}_g$. 
The theta functions obey the \emph{Theta Transformation Formula} \cite{BirkenhakeLange2004}:
\begin{equation}
\theta[\gamma \cdot m](\gamma\cdot (z,  \tau)) = \phi(\gamma,m,z,\tau)\cdot \sqrt{\det(C\tau+D)}  \cdot \theta[m](z,\tau)
\end{equation}	
where $\phi(\gamma,m,z,\tau) \in \mathbb{C}^*$ is an explicit function of the parameters with the same sign ambiguity as $\sqrt{\det(C\tau+D)}$. In particular, for a fixed $\tau$, the symmetric theta divisor $\Theta^{[m]}_{\tau}$ is sent by $\gamma\in \Gamma$ to $\Theta^{\gamma\cdot [m]}_{\gamma\cdot \tau}$, so the universal theta divisor $\{ (z,\tau) \,|\, \theta(z,\tau) = 0 \}$ is not well defined in $\mathcal{X}_g$. To address this issue, we define the subgroups
\begin{align}
\Gamma_g(4) &:= \left\{  \gamma \in \Gamma_g \, | \, \gamma \equiv \mbox{Id} \mod 4  \right\}\,, \\
\Gamma_g(4, 8) &:= \left\{  \begin{pmatrix} A & B \\ C & D \end{pmatrix} \in \Gamma_g(4) \, \mid \, \operatorname{diag}(A^tB) \equiv \operatorname{diag}(C^t B) \equiv 0 \mod 8  \right\}\,.
\end{align}
The group $\Gamma_{g}(4,8)$ is normal of finite index in $\Gamma_g$, so the corresponding quotients 
$\mathcal{A}_g(4,8) := \mathbb{H}_g / \Gamma_g(4,8)$ and $\mathcal{X}_g(4,8):=(\mathbb{C}^g\times \mathbb{H}_g)/(\mathbb{Z}^{2g}\rtimes \Gamma_g(4,8))$ are finite Galois covers of $\mathcal{A}_g$ and $\mathcal{X}_g$ respectively.  Moreover, $\Gamma_g(4,8)$ acts trivially on the set of characteristics, so the Theta Transformation Formula shows that for every characteristic $m$, the universal theta divisor 
\begin{equation}
\Theta^{[m]} = \{ (z,\tau) \,|\, \theta[m](z,\tau) = 0 \} = \bigcup_{\tau \in \mathcal{A}_g(4,8)} \Theta^{[m]}_{\tau} 
\end{equation}
is well-defined in $\mathcal{X}_g(4,8)$. However, we see that the
divisor $2\Theta = \{ 2z \,|\, z\in \Theta \}$ is defined universally
inside $\mathcal{X}_g$. Indeed, for any characteristic $m$, we have $2\Theta_{\tau}^{[m]} = 2\Theta^{[0,0]}_{\tau}$, because $\Theta^{[m]}$ and $\Theta^{[0]}$ differ just by a translation by a two-torsion point. Hence we get a well-defined set
\begin{equation}\label{eq:univtwicetheta}
2\Theta = \{ (2z,\tau) \,|\, \theta(z,\tau)=0 \} = \bigcup_{\tau \in \mathcal{A}_g} 2\Theta_{\tau}
\end{equation}
inside $\mathcal{X}_g$. Finally, it holds that for every $\gamma\in\Gamma_{g}(4,8)$,
\begin{equation}
\theta[m](0,\gamma\cdot \tau) = \sqrt{\det(C\tau+D)} \cdot \theta[m](0,\tau)\,.
\end{equation}
Thus, the even theta constants define a map to projective space,
\begin{equation}
\mathcal{A}_{g}(4,8) \longrightarrow \mathbb{P}^{2^{g-1}(2^g+1)-1}\,, \qquad [\tau] \mapsto [\theta[m](0,\tau)]_{m \text{ even}}
\end{equation} 
which is actually an embedding, and realizes $\mathcal{A}_{g}(4,8)$ as an irreducible quasiprojective variety. By definition, polynomials in the homogeneous coordinates of $\mathbb{P}^{2^{g-1}(2^g+1)}$ correspond to polynomials in the theta constants. 

We would like to point out that, due to the very concrete nature of
the theta functions, much of the geometry of abelian varieties
presented above can be computed explicitly, especially in low genus;
see for example \cite{ChuaKummerSturmfelsSchottky2018,
  RenSamSchraderSturmfelsKummer2013,
  RenSamSturmfelsTropicalizationClassical2014, FreitagSalvatiManniGopel2017}.

\section{Singularities of theta divisors at points of order two}

For a general ppav $(A,\Theta)\in \mathcal{A}_g$, the theta divisor is
smooth. The locus $N_0 \subseteq \mathcal{A}_g$ where the theta
divisor is singular is called the Andreotti-Mayer locus, and it is a
divisor in $\mathcal{A}_g$ \cite{AndreottiMayerSingular1967}. Mumford
\cite{MumfordSiegel1983}, Beauville \cite{BeauvillePrym1977} and
Debarre \cite{DebarreLieu1992} proved that $N_0$ has two irreducible
components $\theta_{\text{null}}$ and $N_0'$. The component $\theta_{\text{null}}$
is the locus where the theta divisor contains an even two-torsion
point of $A$, a so-called \emph{vanishing theta null}, and $N_0'$ is the residual
component. Theta divisors containing even two-torsion points correspond to
period matrices $\tau$ such that $\theta[m](0,\tau)=0$ for some even
characteristic $m$. These points are always singular, since all the derivatives $\frac{\partial \theta[m]}{\partial z_i}$ are odd. Moreover, we can give explicit equations for the lift of $\theta_{\text{null}}$ in $\mathcal{A}_{g}(4,8)$ as
\begin{equation}
\theta_{\text{null}} =\left\{ \tau \mid \theta[m](0,\tau) = 0,  \text{ for some even characteristic } m \right\}\,.
\end{equation} 
Thus, on $\mathcal{A}_g(4,8)$ the divisor $\theta_{\text{null}}$ splits into irreducible components corresponding to the even characteristics, and these components are all conjugate under the action of $\Gamma_g$.

As explained in Section~\ref{section:intro}, the rank of the quadric tangent cone at the vanishing theta null leads to a stratification
\begin{equation}
\theta_{\text{null}}^0 \subseteq \theta_{\text{null}}^1 \subseteq \cdots\ \subseteq \theta^{g-1}_{\tnull} \subseteq \theta^{g}_{\tnull} = \theta_{\tnull}
\end{equation}
where $\theta^h_{\tnull}$ is the locus of ppav whose theta divisor
contain an even two-torsion point with a quadric tangent cone of rank at
most $h$. In $\mathcal{A}_g(4,8)$, the locus $\theta^h_{\tnull}$ is given by the equations
\begin{align}
\begin{split}
\theta^h_{\tnull} &= \left\{ \tau \in \mathcal{A}_g(4,8) \mid \exists\, m \text{ even: } \theta[m](0,\tau) = 0,\,\,\, \operatorname{rk} \left( \frac{\partial^2 \theta[m]}{\partial z_i\partial z_j}(0,\tau)  \right)\leq h  \right\} \\
& = \left\{ \tau \in \mathcal{A}_g(4,8) \mid \exists \, m \text{ even: } \theta[m](0,\tau) = 0,\,\,\, \operatorname{rk} \left( (1+\delta_{ij}) \frac{\partial \theta[m]}{\partial \tau_{ij}}(0,\tau)   \right) \leq h  \right\}
\end{split}
\end{align}
where the second equality comes from the heat equation
\eqref{eq:heatequation}. One should be slightly careful with this
description because, as noted by Grushevsky and Salvati Manni, the
condition $\left\{ \operatorname{rk} \left( (1+\delta_{ij})
\frac{\partial \theta[m]}{\partial \tau_{ij}}(0,\tau)   \right) \leq h
\right\}$ is not well-defined on $\mathcal{A}_g(4,8)$; it is
well-defined only together with the condition $\{\theta[m](0,\tau)=0\}$. An alternative set of equations that fixes this problem is given in \cite{GrushevskySalvatiManniVanishing2008}.

If $C$ is a smooth curve of genus $g$, a vanishing theta null on  the Jacobian $J(C)$ corresponds to an even theta characteristic $\kappa$ such that $h^0(C,\kappa)\geq 2$. As explained in Section~\ref{section:intro},
\begin{equation}
\mathcal{J}_g \cap \theta_{\tnull} \subseteq \theta^3_{\tnull}\,.
\end{equation}  
In \cite{GrushevskySalvatiManniVanishing2008}, it was proved that this
is an equality in genus four, and it was asked whether
$\mathcal{J}_g\cap \theta_{\tnull}$ is an irreducible component of
$\theta^3_{\tnull}$ in higher genus. In the rest of this paper, we will
discuss this problem and present a proof for genus five. 

\section{Partial compactification and Gauss maps}

Our strategy is to bound the dimension of the irreducible component of
$\theta^3_{\tnull}$ which contains $\mathcal{J}_g\cap
\theta_{\tnull}$. To do so, we will use a partial compactification of $\mathcal{A}_g$.

\subsection{Partial compactification  of $\mathcal{A}_g$}\label{section:partialcompactification}

The idea is to follow the strategy of
\cite{GrushevskySalvatiManniThetaTwo2007}, by studying the boundary of
$\theta^3_{\tnull}$ inside the partial compactification
$\overline{\mathcal{A}}_g^1 = \mathcal{A}_g \cup \partial
\mathcal{A}_g$ of $\mathcal{A}_g$. This was introduced by \cite{IgusaDesingularization1967} and then studied by Mumford \cite{MumfordSiegel1983}. The boundary divisor $\partial\mathcal{A}_g$ parametrizes rank one degenerations of ppavs, which we briefly describe here, referring to \cite{MumfordSiegel1983} for more details. Let $(B,\Xi) \in \mathcal{A}_{g-1}$ be a ppav and let $G$ be an algebraic group, which is an extension of $B$ by $\mathbb{C}^*$:
\begin{equation}\label{eq:rank1}
0 \lra \mathbb{C}^* \lra G \lra B \lra 0\,.
\end{equation}
Then we can look at $G$ as a $\mathbb{C}^*$-bundle on $B$, which can
be completed naturally to a $\mathbb{P}^1$-bundle, together with two
sections $B_0$, $B_{\infty}$. These sections can be glued together via translation by a point $b\in B$, and the resulting variety $\overline{G}$ is a limit of abelian varieties. It also carries a divisor $D\subseteq \overline{G}$, which is a limit of theta divisors. Thus, the boundary divisor is a fibration $p\colon \partial\mathcal{A}_g \to \mathcal{A}_{g-1}$, with fiber $B/\operatorname{Aut}(B,\Xi)$ over $(B,\Xi)\in\mathcal{A}_{g-1}$, and the general fiber is $B/\langle \pm 1 \rangle$, the Kummer variety of $B$.
Analytically, points on the boundary can be seen as limits of $g\times g$ period matrices $\tau=\tau(t)$ such that the imaginary part of $\tau_{gg}$ goes to $+\infty$ as $t\to 0$, and all the other coordinates converge. Hence, the limit has the form 
\begin{equation}
\tau = \begin{pmatrix} \tau' & z' \\ z'^t & i\infty \end{pmatrix} 
\end{equation}
where $\tau'$ is the period matrix of $(B,\Xi) \in \mathcal{A}_{g-1}$ and $z'$ represents the translation point $b\in B$.
Hence we see that, at least around a point $(B,\Xi) \in \mathcal{A}_{g-1}$, we have a surjective map 
\begin{equation}\label{eq:maptoboundary}
\mathcal{X}_{g-1} \to \partial\mathcal{A}_{g-1}
\end{equation} that over $(B,\Xi)$ corresponds to $B \to B/\operatorname{Aut}(B)$. 

The boundary of the theta null divisor in the partial compactification was computed by Mumford \cite{MumfordSiegel1983}.

\begin{theorem}[Mumford]\label{thm:boundarynull}
Let $\theta_{\tnull}$ be the closure of the theta null divisor in $\overline{\mathcal{A}}_g^1$. Then 
\begin{equation}
\theta_{\tnull} \cap \partial\mathcal{A}_{g-1} = \left( \bigcup_{(B,\Xi)} 2_{B}(\Xi) \right) \cup p^{-1}(\theta_{\tnull,g-1})
\end{equation}
where $2_B(\Xi) = \{ 2x \,|\, x\in \Xi \}$ is the image of the divisor under the multiplication map. 
\end{theorem} 

More precisely, by $\bigcup 2_B(\Xi)$ we mean the image of the universal double theta divisor of \eqref{eq:univtwicetheta} under the map $\mathcal{X}_{g-1} \to \partial\mathcal{A}_{g-1}$. We denote this component by $X_g$ and sometimes we will not distinguish whether we take it in $\mathcal{X}_{g-1}$ or in $\partial\mathcal{A}_{g-1}$. By \eqref{eq:univtwicetheta} we can write  
\begin{equation}\label{eq_Xg}
X_g = \left\{ (2z',\tau') \,|\, \theta '\left(z',\tau'\right) = 0 \right\}
\end{equation} 
where now $\theta'$ is the Riemann theta function in genus $g-1$.  The intersections of the strata $\theta^h_{\tnull}$ with $X_g$ were determined by Grushevsky and Salvati Manni \cite{GrushevskySalvatiManniThetaTwo2007}.

\begin{theorem}[Grushevsky, Salvati Manni]\label{thm:rankD}
Denote again by $\theta^h_{\tnull}$ the closure in $\overline{\mathcal{A}}_g^1$ of the corresponding stratum in $\mathcal{A}_g$. Define the matrix
\begin{equation}
D\gamma(z',\tau') :=
\begin{pmatrix} 
\frac{\partial^2\theta'}{\partial {z'_1}^2} & \frac{\partial^2\theta'}{\partial z'_1\partial z_2'} & \cdots & \frac{\partial \theta'}{\partial z'_1 \partial z'_{g-1}} & \frac{\partial\theta'}{\partial z'_1} \\   
\frac{\partial^2\theta'}{\partial z'_1\partial z'_2} & \frac{\partial^2\theta'}{\partial {z'_2}^2} & \cdots & \frac{\partial \theta'}{\partial z'_2 \partial z'_{g-1}} & \frac{\partial\theta'}{\partial z'_2} \\ 
\vdots & \vdots & \ddots & \vdots & \vdots \\
\frac{\partial^2\theta'}{\partial z'_1\partial z'_{g-1}} & \frac{\partial^2\theta'}{\partial {z'_2}\partial{z'_{g-1}}} & \cdots & \frac{\partial \theta'}{\partial {z'}^2_{g-1}} & \frac{\partial\theta'}{\partial z'_{g-1}} \\
\frac{\partial \theta'}{\partial z'_1} & \frac{\partial \theta'}{\partial z'_2} & \cdots & \frac{\partial \theta'}{\partial z'_{g-1}} & 0
\end{pmatrix} 
\end{equation}
evaluated at $(z',\tau')$. Then
\begin{equation}
\theta^h_{\tnull} \cap X_g = \left\{ (2z',\tau') \,|\,  \theta'\left(z',\tau'\right) = 0, \,\,\, 
\operatorname{rk} D\gamma\left(z',\tau'\right) \leq h
\right\}\,.
\end{equation}
\end{theorem}

\subsection{Thom-Boardman loci of the Gauss map}
 
 The last result is best understood in terms of the Gauss map for the theta divisor. If $(B,\Xi) \in\mathcal{A}_{g-1}$ is a ppav, then the Gauss map is the rational map
 \begin{equation}
 \gamma\colon \Xi \dashrightarrow \mathbb{P}(T_0B) 
 \end{equation}
 that associates to each smooth point $p\in \Xi$ the tangent space $T_p\Xi$, seen as a subspace of $T_0B$ via translation by $p$. The base locus of the map is precisely the singular locus of $\Xi$, and if  $\Xi=\{ \theta'(z',\tau') = 0 \}$, then the Gauss map can be written explicitly as
 \begin{equation}\
 \gamma = \left[ \frac{\partial \theta'}{\partial z'_1}\,, \ldots\,, \frac{\partial \theta'}{\partial z_{g-1}} \right]\,.
 \end{equation}
 
 The Gauss map is related to Theorem \ref{thm:rankD} by the following
 result, whose proof is the same as \cite[Lemma
   2]{GrushevskySalvatiManniThetaTwo2007}; see also
 \cite{deJongTheta2010, RohrbachMaster2014}.
 
 \begin{lemma}\label{lemma:rankD}
 Let $(B,\Xi)\in \mathcal{A}_{g-1}$ be a ppav,  with the theta divisor given as $\Xi = \{ \theta'(z',\tau') = 0 \}$. If $z'\in \Xi$ is a smooth point, then
 \begin{equation}
 \operatorname{rk} D\gamma(z',\tau') \leq h \qquad \text{ if and only if } \qquad  \operatorname{rk } d\gamma_{z'} \leq h-2
 \end{equation}
 where $d\gamma_{z'}$ is the differential of the Gauss map at $z'$. If instead $z'\in \Xi$ is a singular point,
 \begin{equation}
 \operatorname{rk} D\gamma(z',\tau') \leq h \qquad \text{ if and only if } \qquad  \operatorname{rk } Q_{z'}\Xi \leq h
 \end{equation}
 where $Q_{z'}\Xi$ is the quadric tangent cone to $\Xi$ at $z'$.
 \end{lemma}
 
For a map $f\colon X\to Y$ between two smooth connected varieties, the (closed) \emph{Thom-Boardman loci} are defined as
\begin{align}
\Sigma^i(f) &:= \{ p\in X \,|\, \dim \ker df_p \geq i \}\,,\\
\Sigma_i(f) &:= \Sigma^{\dim X-i}(f) = \{ p\in X \,|\, \operatorname{rk} df_p \leq i  \}\,.
\end{align} 

Hence we can interpret Lemma \ref{lemma:rankD} as saying that for a theta divisor $\Xi_{\tau'} = \{ \theta'(z',\tau') = 0 \}$, the locus of smooth points where $\operatorname{rk} D\gamma(z',\tau') \leq h$ coincides with the Thom-Boardman locus $\Sigma_{h-2}(\gamma_{\tau'})$ of the Gauss map $\gamma_{\tau'}\colon \Xi_{\tau'}^{sm} \to \mathbb{P}^{g-2}$. Moreover, Lemma \ref{lemma:rankD} suggests also that the Thom-Boardman loci can be naturally extended to a closed subset of $\Xi$ by
\begin{equation}
\overline{\Sigma}_{h-2}(\gamma_{\tau'}) = \Sigma_{h-2}(\gamma_{\tau'}) \cup \{ z' \in \Xi^{sing} \,|\, \operatorname{rk} Q_{z'}\Xi_{\tau'} \leq h \}\,.
\end{equation} 
Thus Theorem \ref{thm:rankD} can be rephrased as saying that
\begin{equation}\label{eq:rankDTB}
\theta^h_{\tnull} \cap X_g = \bigcup_{(B,\Xi) \in\mathcal{A}_{g-1}} 2_{B}(\overline{\Sigma}_{h-2}(\gamma_\Xi))\,,
\end{equation} 
where the set on the right can be viewed as the universal second
multiple of the Thom-Boardman locus inside the universal family
$\mathcal{X}_{g-1} \to  \mathcal{A}_{g-1}$. Observe that this is well-defined in $\mathcal{X}_{g-1}$ because of the multiplication by two.

\section{Jacobians with a vanishing theta null in genus five}

Now we turn to the proof of our main theorem. To begin with, we
observe that $\mathcal{J}_g\cap \theta_{\tnull}$ is the image under
the Torelli map of the locus of curves with a vanishing theta null.  It was
proved by Teixidor \cite{TeixidoriThetanull1988} that this is an
irreducible divisor in $\mathcal{M}_g$, and since the Torelli map is
injective, we see that $\mathcal{J}_g\cap \theta_{\tnull}$  is
irreducible of dimension $3g-4$. What we need to prove is that the
irreducible component $\mathcal{Z}_g$ of $\theta^3_{\tnull}$ which
contains $\mathcal{J}_g\cap \theta_{\tnull}$ has the same
dimension. Equivalently, we can compute the codimension inside the
divisor $\theta_{\tnull}$, and we expect that
\begin{equation}
\operatorname{codim}_{\theta_{\tnull}}(\mathcal{Z}_g) \geq \operatorname{codim}_{\theta_{\tnull}}(\mathcal{J}_g\cap\theta_{\tnull}) = \frac{g(g+1)}{2} -1 - (3g-4) = \frac{1}{2}(g-2)(g-3)\,. 
\end{equation} 
We can bound the codimension of $\mathcal{Z}_g$ by considering its intersection with $X_g$. More precisely, we proceed in two steps.
\begin{enumerate}
\item Prove that $\mathcal{J}_g\cap \theta_{\tnull}$ intersects $X_g$
  at a smooth point of $\theta_{\tnull}$. Then $\mathcal{Z}_g$ also
  intersects $X_g$ at a smooth point of $\theta_{\tnull}$, and it is a standard fact that $\operatorname{codim}_{\theta_{\tnull}}\mathcal{Z}_g \geq \operatorname{codim}_{X_g}(\mathcal{Z}_g\cap X_g)$.

\item Show that $\operatorname{codim}_{X_g}(\theta^{3}_{\tnull} \cap
  X_g) \geq \frac{1}{2}(g-2)(g-3)$. Then $\operatorname{codim}_{X_g}
  (\mathcal{Z}_g\cap X_g)\geq \frac{1}{2}(g-2)(g-3)$, and from the
  previous step we have
  $\operatorname{codim}_{\theta_{\tnull}}\mathcal{Z}_g \geq
  \frac{1}{2}(g-2)(g-3)$ as well. 
\end{enumerate}

We proceed to carry out these steps in genus five.

\subsection{Limits of Jacobians in the partial compactification}

Here we consider points in the intersection $\mathcal{J}_g\cap
\partial\mathcal{A}_g$ coming from nodal curves. Let $C$ be a smooth
and irreducible curve of genus $g-1\geq 4$ with two distinct marked points
$p,q\in C$, and let $\overline{C}$ be the nodal curve obtained by identifying these two points. This is a stable curve of arithmetic genus $g$ and the map $\nu\colon C\to \overline{C}$ is the normalization. 

The group $\operatorname{Pic}^0(\overline{C})$ fits naturally into an exact sequence \cite{ACGH2}
\begin{equation}
0 \lra \mathbb{C}^* \lra \operatorname{Pic}^0(\overline{C}) \overset{\nu^*}{\lra} J(C) \lra 0
\end{equation}
where $J(C)=\operatorname{Pic}^0(C)$ is the Jacobian of $C$. This says
that a bundle $\overline{L}$ on $\overline{C}$ is equivalent to a
bundle $L$ on $C$, together with an identification of the two fibers
$L(p),L(q)$, which is given by an element in $\mathbb{C}^*$. Hence we see that $\operatorname{Pic}^0(\overline{C})$ is a rank one extension of the ppav $J(C)$, as in \eqref{eq:rank1}. Moreover, in this situation we also have a natural element in $J(C)$ given by $\mathcal{O}(p-q)$. According to the description of Section \ref{section:partialcompactification}, this gives us an element in $\mathcal{J}_g\cap \partial\mathcal{A}_g$. 

Now we consider nodal curves with a vanishing theta null, following Cornalba's construction \cite[Example 3.2, Example 6.2]{CornalbaModulitheta1989}.

\begin{lemma}\label{lemma:bundleL}
Let $L$ be a line bundle on $C$ such that $2L \sim K_C+p+q$. Then $L$ descends to a line bundle $\overline{L}$ on $\overline{C}$ such that $2\overline{L} \sim K_{\overline{C}}$ and $h^0(\overline{C},\overline{L}) = h^0(C,L)$.
\end{lemma}

\begin{proof}
We sketch the proof here, referring to \cite{CornalbaModulitheta1989}
for more details. First we observe that $\deg L = g-1$, and the
Riemann-Roch theorem shows that $h^0(C,L-p-q) = h^0(C,L)-1$, so the
linear system of $L$ does not separate the two points $p,q$. Moreover,
this shows that there is a section $\sigma_0\in H^0(C,L)$ that does
not vanish at  either $p$ or $q$; in fact $\sigma_0$ does not vanish
at both points. Indeed, suppose that $\sigma_0(q)=0$ and
$\sigma_0(p)\ne 0$. Then $\sigma_0^2$ is a section of $2L(-q)\sim  K_C
+ p-q$ that does not vanish at $p$, which cannot happen because $p$ is a base point of $K_C+p-q$, since $C$ has positive genus.

Hence, we can identify the two fibers of $L(p)$, $L(q)$ by identifying
$\sigma_0(p)$ with $\sigma_0(q)$. This induces a line bundle
$\overline{L}$ on $\overline{C}$ such that $2\overline{L}
\sim K_{\overline{C}}$, and the sections
$H^0(\overline{C},\overline{L})$ correspond to sections $\sigma \in
H^0(C,L)$ such that $\sigma(p)=\sigma(q)$. Moreover, we know that
$H^0(C,L)$ does not separate $p$, $q$, which means precisely that $H^0(\overline{C},\overline{L}) = H^0(C,L)$.    
\end{proof} 

According to \cite{CornalbaModulitheta1989}, this gives a theta
characteristic on $\overline{C}$, and if $h^0(C,L)=2$, we get that
$h^0(\overline{C},\overline{L})=2$, so this is an even and effective
theta characteristic on $\overline{C}$. This actually gives an element in
$(\mathcal{J}_g\cap\theta_{\tnull})\cap X_g$ as follows.

\begin{lemma}\label{lemma:L2}
Let $L$ be a line bundle on $C$ such that $2L \sim K_C+p+q$ and $h^0(C,L)=2$. Then $(J(C),\mathcal{O}(p-q))$ gives a point in $(\mathcal{J}_g\cap\theta_{\tnull})\cap X_g$.
\end{lemma}

\begin{proof} 
According to the description of Section \ref{section:partialcompactification}, we need to show that $\mathcal{O}(p-q) \in 2\cdot \Theta$, where $\Theta\subseteq J(C)$ is a symmetric theta divisor. Recall that all the symmetric theta divisors on $J(C)$ are of the form $W_{g-2}(C)-\kappa$, where $W_{g-2}(C) = \{ M\in \operatorname{Pic}^{g-2}(C) \,|\, h^0(C,M)>0 \}$, and $\kappa$ is a theta characteristic, hence $2\cdot \Theta = 2\cdot W_{g-2}(C) - K_C$.  Then we observe that $L(-q) \in W_{g-2}(C)$ and moreover $2L(-q)\sim K_C + p-q $, so $p-q \sim 2L(-q)-K_C \in  2\cdot W_{g-2}(C) - K_C$. 
\end{proof} 

Via this construction, we obtain many points in $(\mathcal{J}_g\cap \theta_{\tnull})\cap X_g$ and we can give conditions for these to be smooth points of $\theta_{\tnull}$. This is done in the next two lemmas.

\begin{lemma}\label{lemma:smoothXg}
Let $(B,\Xi)\in \mathcal{A}_{g-1}$ be such  that $\operatorname{Aut}(B,\Xi)= \{ \pm 1 \}$ and let $b\in 2_{B}(\Xi)$ be a point such that
 \begin{enumerate}
 	\item $b$ is not a two-torsion point.
 	\item The set $2_B^{-1}(b) \cap \Xi$ consists precisely of one point $b'$.
 	\item The point $b'$ has multiplicity at most two in $\Xi$.
 \end{enumerate}
  Then $(B,\Xi)$ and $b$ give a smooth point of $X_g$ which is also a smooth point of $\theta_{\tnull}$.   	
\end{lemma}
\begin{proof}
First observe that $X_g$ is a component of the intersection of the two divisors $\theta_{\tnull}$ and $\partial\mathcal{A}_{g}$, so a smooth point in $X_g$ is smooth in both $\theta_{\tnull}$ and $\partial\mathcal{A}_{g}$. Hence it is enough to prove smoothness in $X_g$.

To do so, we look at the point $(b,(B,\Xi))$ in the universal abelian
variety $\mathcal{X}_{g-1}$; this gives a point in
$\partial\mathcal{A}_{g}$ via the map $\mathcal{X}_{g-1}\to
\partial\mathcal{A}_{g}$ of \eqref{eq:maptoboundary}. Since we are
assuming that $(B,\Xi)$ has no extra automorphisms and that $b$ is not
a two-torsion point, we see that the map is a local isomorphism around
$(b,(B,\Xi))$, hence we can work in $\mathcal{X}_{g-1}$ instead of in
$\partial\mathcal{A}_{g}$. Moreover, the fact that $(B,\Xi)$ has no
extra automorphisms also tells us that the map
$\mathcal{X}_{g}(4,8)\to\mathcal{ X}_g$ is a local isomorphism around
$(b,(B,\Xi))$. So we can work directly inside $\mathcal{X}_g(4,8)$,
where we can look at $X_g = \bigcup_{(B,\Xi)} 2_B(\Xi)$ as the image
of the universal theta divisor $\Xi_{g-1} = \bigcup_{(B,\Xi)} \Xi$
under the global multiplication-by-two map $2\colon
\mathcal{X}_{g-1}(4,8) \to \mathcal{X}_{g-1}(4,8)$. 

Now, our second assumption on $b$ shows that the fiber of this map
over $(b,(B,\Xi))$ consists of a single point $(b',(B,\Xi))$. We also
know that the differential of the multiplication-by-$2$ map is just
the usual scalar multiplication by $2$, so it is an isomorphism. Hence,
the tangent space to $X_g$ at $(b,(B,\Xi))$ is isomorphic to the
tangent space to the universal theta divisor $\Xi_{g-1}$ at
$(b',(B,\Xi))$. In local coordinates $(z',\tau')$, the universal theta
divisor is given by $\{\theta(z',\tau')=0\}$, and the heat equation
shows that a point $(z',\tau')$ is singular if and only if $z'$ is a singular point of $\Xi_{\tau'}$ of multiplicity at least $3$, which in our case is ruled out by the third assumption. 
\end{proof}

\begin{lemma}\label{lemma:smoothpoint}
Let $C$ be a smooth curve of genus $g-1\geq 4$ with only trivial automorphisms and let $D\in C_{g-3}$ be an effective divisor such that
\begin{enumerate}
\item $h^0(C,D)=1$.\label{eq:smoothpointcondition1}
\item $K_C-2D\sim p+q$ with $p,q$ two distinct points on $C$. \label{eq:smoothpointcondition2}
\item $h^0(C,\eta+D)=0$ for each $\eta\in\operatorname{Pic}^0(C)[2] \setminus \{\mathcal{O}_C\}$. \label{eq:smoothpointcondition3}
\end{enumerate}
Then $J(C)$ together with the point $\mathcal{O}(p-q)$ give a point in $(\mathcal{J}_g\cap\theta_{\tnull})\cap X_g$ which is a smooth point of $\theta_{\tnull}$.
\end{lemma}
\begin{proof}
First we observe that $\deg(K_C-2D) = 2(g-1)-2-2(g-3) = 2$, and since the curve is not hyperelliptic, this implies $h^0(C,K_C-2D)\leq 1$, so the points $p,q$ in \eqref{eq:smoothpointcondition2} are uniquely determined. 	

We then apply Lemma \ref{lemma:smoothXg}. Since the automorphism group
of $C$ is trivial, Torelli's theorem implies that $J(C)$ has no extra
automorphisms. Moreover, if $\kappa$ is any theta characteristic on
$C$, we can fix the symmetric theta divisor given by
$W_{g-2}(C)-\kappa$, and twice the theta divisor is
$2W_{g-2}(C)-K_C$. Then $\mathcal{O}(p-q) \in 2W_{g-2}(C)-K_C$,
since if we set  $L = \mathcal{O}(D+p+q)$, then $2L\sim K_C+p+q$
because of assumption \eqref{eq:smoothpointcondition2}. Reasoning as in the proof of Lemma \ref{lemma:bundleL}, we get that $h^0(C,L) = h^0(C,D)+1 = 2$, and also $h^0(C,L(-p)) = h^0(C,L(-q)) = h^0(C,L)-1=1$. Hence $p-q \sim 2L(-q) - K_C$, with $L(-q)\in W_{g-2}(C)$.

Now we check the three conditions of Lemma \ref{lemma:smoothXg}.  We see that $\mathcal{O}(p-q)$ cannot be a two-torsion point, because otherwise $2p\sim 2q$, but $p,q$ are distinct and $C$ is not hyperelliptic. For the second condition, since $p-q \sim 2L(-q)-K_C$, we see that 
\begin{equation}\label{eq:intersection}
2_{J(C)}^{-1}(p-q)\cap (W_{g-2}(C)-\kappa)\cong \{ L(-q)+\eta \,|\, \eta \in J(C)[2],\, h^0(L+\eta-q) > 0 \}\,.
\end{equation}
For any $\eta\in\operatorname{Pic}^0(C)[2]$, we see that $2(L+\eta) \sim 2L \sim K_C+p+q$, hence from the proof of Lemma \ref{lemma:bundleL}, $h^0(C,L+\eta-q)=h^0(C,L+\eta)-1 = h^0(C,L+\eta-p-q) = h^0(C,D+\eta)$. Thus assumption \eqref{eq:smoothpointcondition3} implies that the intersection \eqref{eq:intersection} consists of the unique point $L(-q)$.

For the last condition, we need to check that $L(-q)$ has multiplicity at most two in the theta divisor $W_{g-2}(C)$. But Riemann's singularity theorem shows that $\operatorname{mult}_{L(-q)} W_{g-2}(C) = h^0(C,L-q) = h^0(C,L)$, and $h^0(C,L)=1$ by assumption \eqref{eq:smoothpointcondition1}. 
\end{proof}

We now specialize this discussion to $g=5$.

\begin{proposition}\label{prop:smoothpoint5}
	$\mathcal{J}_5\cap \theta_{\tnull}$ and $X_5$ intersect at a smooth point of $\theta_{\tnull}$.
\end{proposition}
\begin{proof}
We follow Lemma \ref{lemma:smoothpoint}. Let $C$ be a general curve of
genus $4$, with only trivial automorphisms. We consider the locus 
\begin{equation} 
Z=\{ D\in  C_2 \,|\, h^0(C,K_C-2D)>0 \}
\end{equation}
and we show that a general element in $Z$ satisfies the three conditions of Lemma \ref{lemma:smoothpoint}.
First, we compute the dimension of $Z$. The condition $h^0(C,K_C-2D)>0$ is equivalent to the fact that the evaluation map $\operatorname{ev}_{D}\colon H^0(C,K_C)\to H^0(C,K_C\otimes \mathcal{O}_{2D})$ is not injective, or not an isomorphism, since $h^0(C,K_C) = h^0(C,K_C\otimes \mathcal{O}_{2D}) = 4$.   
The evaluation map globalizes to a map of rank four bundles
$\operatorname{ev}\colon H^0(C,K_C)\otimes \mathcal{O}_{C_2} \to E$ on
$C_2$, where the fiber of $E$ at $D\in C_2$ is precisely
$H^0(C,K_C\otimes \mathcal{O}_{2D})$. Thus, $Z$ is precisely the
degeneracy locus of the global evaluation map. In particular, if we
can show that $Z\ne\emptyset$ and $Z\ne C_2$, it will follow that $Z$ is a divisor in $C_2$. To prove this, let $x\in C$ be a general point, then $h^0(C,K_C-2x) = 2$, so after removing the eventual base locus $B$, we get a map $\varphi\colon C\to \mathbb{P}^1$. If $y$ is a point where $\varphi$ is ramified, we see that $D=x+y\in Z$; if instead $y'$ is a point outside of $B$ where $\varphi$ is not ramified, then $D'= x+y'\notin Z$.

Now we check the conditions in Lemma \ref{lemma:smoothpoint}. If $D\in
Z$, then $h^0(C,D)=1$, otherwise the curve would be hyperelliptic. For the second condition, suppose that $K_C-2D \sim 2p$ for some $p\in C$. Then $K_C \sim 2D+2p$, so $\mathcal{O}_C(D+p)$ is a theta characteristic on $C$. There are only finitely many effective theta characteristics $\kappa$ and since $C$ is general, they all satisfy $h^0(C,\kappa)=1$. Hence there are only finitely many such $D$, and since $\dim Z=1$, a general $D$ in any component of $Z$ will give $K_C-2D \sim p+q$ for two distinct points $p,q\in C$.

For the third condition, if we look at $C_2$ in
$\operatorname{Pic}^2(C)$ via the Abel-Jacobi map, we see that $D\in
Z$ satisfies the third condition precisely when $D\notin
\bigcup_{\eta}(C_2+\eta)$, as $\eta$ varies in $J(C)[2]\setminus
\{\mathcal{O}_C\}$. Suppose by contradiction that $Z\subseteq
\bigcup_{\eta}(C_2+\eta)$ and let $Z'\subseteq Z$ be an irreducible
component. Then there must be an $\eta\in J(C)[2]\setminus \{
\mathcal{O}_C\}$ such that $Z'\subseteq C_2+\eta$, so that
$Z'\subseteq C_2\cap(C_2+\eta)$. In particular $\dim
(C_2\cap(C_2+\eta))\geq 1$, and we will now prove that this cannot happen. To do so, observe that if $\dim (C_2\cap(C_2+\eta))\geq 1$, then the difference map
\begin{equation}
\phi_2\colon C_2\times C_2 \lra \operatorname{Pic}^0(C)\,, \qquad (D,E)\mapsto D-E
\end{equation} 
has a positive-dimensional fiber at $\eta$, and by \cite[Exercise
  V.D-4]{ACGH} this is possible only if $\eta \cong \mathcal{O}(x-y)$
for two distinct points $x,y\in C$, or if $\eta \cong M_1-M_2 \cong
2M_1-K_C$, where $M_1$ is one $\mathfrak{g}^{1}_3$ on $C$ and $M_2 =
K_C-M_1$ is the other. The first possibility cannot happen because
then $2x\sim 2y$, and $C$ is not hyperelliptic. For the second
possibility, observe that if $2M_1-K_C$ is two-torsion, then $4M_1\sim
2K_C$. Since $C$ is not hyperelliptic, we know from Noether's theorem
that it is projectively normal in the canonical embedding $C\subseteq
\mathbb{P}^{3}$, hence $4M_1\sim 2K_C$ if and only if for every divisor $H\in |M_1|$, the divisor $4H$ is cut out by a quadric in $\mathbb{P}^3$. Using Macaulay2 \cite{M2}, we can easily find a curve where this does not happen, for example the curve $C = \{ X_0X_3-X_1X_2 = X_0^3+X_0^2X_1+X_1^3+X_2^3+X_3^3 = 0 \}$ and the divisor $H=C\cap\{X_2=X_3=0\}$.
\end{proof}

\subsection{Thom-Boardman loci for the Gauss map in dimension 4}

Now  we need to bound the codimension of $\theta^{3}_{\tnull} \cap X_g$ inside $X_g$, and by \eqref{eq:rankDTB}, we can do it by studying the Thom-Boardman loci in one dimension less. 

\begin{lemma}\label{lemma:boundcodimgeneral}
Suppose that 
\begin{equation}
\operatorname{codim}_{\mathcal{A}_{g-1}} \{ (B,\Xi) \,|\, \dim \overline{\Sigma}_1(\gamma_{\Xi}) \geq i \} \geq i + \frac{(g-2)(g-5)}{2}\,, \qquad \text{ for } i=0,\ldots,g-2\,.
\end{equation}
Then
\begin{equation}
\operatorname{codim}_{X_g}(\theta^{3}_{\tnull} \cap X_g) \geq \frac{1}{2}(g-2)(g-3)\,.
\end{equation}
\begin{proof}
We know from \eqref{eq:rankDTB} that $\theta^3_{\tnull}\cap X_g = \bigcup_{(B,\Xi)} 2_B(\overline{\Sigma}_1(\gamma_{\Xi}))$, and we can bound the dimension of this by bounding the dimension of the fibers and of the image along the projection $p\colon \theta^3_{\tnull}\cap X_g \to \mathcal{A}_{g-1}$. More precisely we have
\begin{equation}
\dim \theta^3_{\tnull}\cap X_g \leq \max_{i=0,\ldots,g-2} \left( i + \dim \{ (B,\Xi) \in \mathcal{A}_{g-1} \,|\, \dim 2_B(\overline{\Sigma}_1(\gamma_{\Xi})) \geq i \} \right)\,.
\end{equation}
Since the multiplication $2_B\colon B\to B$ is a finite map, we can replace $\dim 2_B(\overline{\Sigma}_1(\gamma_{\Xi})) \geq i$ by $\dim \overline{\Sigma}_1(\gamma_{\Xi})\geq i$, and if we rephrase the previous inequality in terms of the codimension, we get exactly what we want.  	
\end{proof}
\end{lemma}

\begin{remark}\label{rmk:alwaysbadpoints}
When $g>5$, Lemma \ref{lemma:boundcodimgeneral} requires that the set
\begin{equation}
\{ (B,\Xi) \in \mathcal{A}_{g-1} \,|\, \dim \overline{\Sigma}_1(\gamma_{\Xi}) \geq 0  \} = \{ (B,\Xi) \in \mathcal{A}_{g-1} \,|\, \overline{\Sigma}_1(\gamma_{\Xi}) \ne \emptyset  \}
\end{equation}
is a proper subset of $\mathcal{A}_{g-1}$. However, this cannot
happen; indeed, if $(B,\Xi)\in \mathcal{A}_{g-1}$ is a general ppav,
then $\Xi$ contains $2^{g-1}(2^{g}-1)$ odd two-torsion points,
corresponding to the odd characteristics. If $\Xi$ is cut out by a
theta function $\Xi = \{\theta'(z',\tau')=0\}$ and $z'$ is an odd
two-torsion point, then one can see that $\operatorname{rk} D\gamma(z',\tau') \leq 2$, so $z'\in \Sigma_1(\gamma_{\Xi})$.  
\end{remark}

However, when $g=5$ the hypotheses of Lemma~\ref{lemma:boundcodimgeneral} are satisfied.

\begin{proposition}\label{prop:boundcodim5}
It holds that
\begin{equation}
\operatorname{codim}_{X_5} (\theta^3_{\tnull} \cap X_5) \geq 3\,.
\end{equation}
\end{proposition} 
\begin{proof}
According to Lemma \ref{lemma:boundcodimgeneral}, we need to prove that 
\begin{equation}
\operatorname{codim}_{\mathcal{A}_{4}} \{ (B,\Xi) \,|\, \dim \overline{\Sigma}_1(\gamma_{\Xi}) \geq i \} \geq i\,,  \qquad \text{ for } i=0,1,2,3\,.
\end{equation}
\begin{itemize} 
\item For $i=0$, this is immediate, though as observed in Remark \ref{rmk:alwaysbadpoints}, it is crucial that $g=5$.
\item For $i=1$, we need to show that for a general $(B,\Xi)\in
  \mathcal{A}_4$, the Thom-Boardman locus
  $\overline{\Sigma}_1(\gamma_{\Xi})$ is finite dimensional. This is proved by Adams, McCrory, Shifrin and Varley in \cite{AdamsetalGauss1993}, where they show that the locus consists exactly of the odd two-torsion points of Remark \ref{rmk:alwaysbadpoints}.
\end{itemize}
For the cases where $i>1$, we start with a general observation. Assume
that the Gauss map $\gamma_{\Xi}\colon \Xi^{sm} \to \mathbb{P}^3$ has
finite fibers, then $\dim \Sigma_1(\gamma_{\Xi}) \leq 1$. Indeed, let
$Z\subseteq \Sigma_1(\gamma_{\Xi})$ be an irreducible component of
positive dimension. By construction, the differential of the induced map ${\gamma_{\Xi}}_{|Z}\colon Z\to \mathbb{P}^3$ has rank at most one at a general point of $Z$, so $\dim \gamma_{\Xi}(Z) \leq 1$ by generic smoothness. Since the map has finite fibers, it follows that $\dim Z\leq 1$. Hence if $i>1$, we can have $\dim \overline{\Sigma}_1(\gamma_{\Xi}) \geq i$ if and only if the locus $\{ z'\in \Xi^{sing} \,|\, \operatorname{rank} Q_{z'}\Xi \leq 3  \}$ has dimension at least $i$.
\begin{itemize}
	\item For $i=2$, let $(B,\Xi) \in \mathcal{A}_4$ be such that $\Xi$ is smooth. Then the Gauss map $\gamma_{\Xi} \colon \Xi \to \mathbb{P}^3$ is finite and the singular locus is empty, so the previous observation implies $\dim \overline{\Sigma}_1(\gamma_{\Xi})\leq 1$. This proves that the locus $ \{ (B,\Xi) \,|\, \dim \overline{\Sigma}_1(\gamma_{\Xi}) \geq 2 \}$  is contained in the divisor $N_{0,4} = \{ (B,\Xi) \,|\, \Xi^{sing} \ne \emptyset \}$. This divisor has two irreducible components, which are exactly the theta null divisor $\theta_{\tnull,4}$ and the Jacobian locus $\mathcal{J}_4$ \cite{BeauvillePrym1977}. Now, let $(B,\Xi)$ be the Jacobian of a hyperelliptic curve of genus $4$. Then $(B,\Xi) \in \theta_{\tnull,4}\cap \mathcal{J}_4$ \cite{BeauvillePrym1977}, and since it is a Jacobian, we know that the Gauss map $\gamma_{\Xi}\colon \Xi^{sm} \to \mathbb{P}^3$ has finite fibers \cite{AndreottiTorelli1958}. Moreover, the singular locus of $\Xi$ has dimension $1$, so the previous observation gives $\dim \overline{\Sigma}(\gamma_{\Xi}) \leq 1$. This proves that the locus $ \{ (B,\Xi) \,|\, \dim \overline{\Sigma}_1(\gamma_{\Xi}) \geq 2 \}$ does not contain the intersection $\theta_{\tnull,4}\cap \mathcal{J}_4$; in particular, it does not coincide with any of the two components, so it has codimension at least two.
	\item For $i=3$, let $(B,\Xi)\in  \mathcal{A}_4$ be such that $\Xi$ is irreducible. Then $\dim \overline{\Sigma}_1(\gamma_{\Xi}) \geq 3$ if and only if $\overline{\Sigma}_1(\gamma_{\Xi}) = \Xi$, but this is impossible because the Gauss map is dominant, so its differential at a general point is an isomorphism. Hence, the locus $ \{ (B,\Xi) \,|\, \dim \overline{\Sigma}_1(\gamma_{\Xi}) \geq 3 \}$ is contained in the locus of decomposable abelian varieties, which has codimension $3$.
\end{itemize}
\end{proof}

\subsection{Conclusion}

We can finally prove our main theorem, and we rewrite the argument here for clarity.

\begin{proof}[Proof of Theorem \ref{thm:mainthm}]
We want to show that $\mathcal{J}_5\cap \theta_{\tnull}$ is an
irreducible component of $\theta^3_{\tnull}$. As remarked before, we
know that $\mathcal{J}_5\cap \theta_{\tnull}$ is irreducible of
dimension $11$, thus if $\mathcal{Z}_5\subseteq \theta^3_{\tnull}$ is
the irreducible component containing $\mathcal{J}_5\cap
\theta_{\tnull}$, we need to show that $\dim \mathcal{Z}_5\leq 11$, or
equivalently, $\operatorname{codim}_{\theta_{\tnull}} \mathcal{Z}_g
\geq 3$. We consider the intersection with $X_5$ inside the partial
compactification $\overline{\mathcal{A}_5}^1$; we know from
Proposition \ref{prop:smoothpoint5} that $\mathcal{Z}_5$ and $X_5$
intersect in a smooth point of $\theta_{\tnull}$, hence
$\operatorname{codim}_{\theta_{\tnull}}\mathcal{Z}_5 \geq
\operatorname{codim}_{X_5}(\mathcal{Z}_5\cap X_5)$. It is enough to
prove that this is at least $3$, or more generally, that
$\operatorname{codim}_{X_5}(\theta^3_{\tnull}\cap X_5) \geq 3$. This is precisely the bound of Proposition \ref{prop:boundcodim5}.
\end{proof}

\begin{example}\label{example:computations}

As we have observed in Section~\ref{section:intro}, this solution to the Schottky problem is effective, since the condition of having a vanishing theta null with a quadric cone of rank at most three can be checked explicitly. We present here an explicit example, where we use the Julia package for theta functions presented in \cite{AgostiniChuaNumerical2019}. Consider the period matrix $\tau \in \mathbb{H}_5$ given by
\begin{scriptsize} 
\begin{equation*}
\begin{pmatrix} 
0.402434+0.684132i  &  -0.181379+0.21894i  &    0.24323-0.134164i &
0.0040306+0.0508533i &  -0.318179+0.143837i\\
-0.181379+0.21894i  &    0.279144+1.01836i &   -0.0979857+0.462218i &
-0.0656562+0.609593i &    -0.14647+0.370062i\\
0.24323-0.134164i &  -0.0979857+0.462218i &    0.166629+0.681357i &
-0.286061+0.0203797i &   0.185582-0.150607i \\
0.0040306+0.0508533i &  -0.0656562+0.609593i &   -0.286061+0.0203797i &
0.0413648+1.4056i &      0.190249+0.828849i\\
-0.318179+0.143837i &    -0.14647+0.370062i &   0.185582-0.150607i &
0.190249+0.828849i &    0.74873+1.01168i
\end{pmatrix} 
\end{equation*}
\end{scriptsize}
We can computationally check that the theta constant with even characteristic 
\begin{equation}
m = \begin{bmatrix} 1 & 0 & 0 & 1 & 0 \\ 1 & 0 & 1 & 1 & 0
 \end{bmatrix} 
\end{equation}
vanishes at $\tau$. Moreover, we can compute the corresponding Hessian matrix
\begin{scriptsize}
\begin{equation*}
\begin{pmatrix}
 -2.79665+5.29764i & -9.57825-9.04671i  &  7.36305+2.28697i &
7.58338+5.34729i &   6.15667-1.90199i \\
-9.57825-9.04671i &  18.9738+8.34582i &   -23.1027-3.10545i &
-9.31944-0.822821i &  0.524289-3.64991i \\
7.36305+2.28697i & -23.1027-3.10545i  &  16.8441-1.15986i &
13.9363-4.56541i &  -3.32248+4.10698i \\
7.58338+5.34729i & -9.31944-0.822821i &  13.9363-4.56541i &
2.89309+1.21773i  &  3.86617-0.546202i \\
6.15667-1.90199i & 0.524289-3.64991i  & -3.32248+4.10698i &
3.86617-0.546202i & -12.9726-1.928i \\
\end{pmatrix}
\end{equation*}	
\end{scriptsize}
The Hessian has the following eigenvalues:
\begin{align*}
47.946229109152995 &+ 9.491932144035298i \\
-15.491689246713147 &+ 3.3401255907497958i \\
-9.512858919129267 &- 1.0587349322052013i \\
-2.7271385943272036\times 10^{-15} &- 1.1117459994936022i \times 10^{-14} \\
-5.698014266322794 \times 10^{-15} &+ 6.342925068807627i \times 10^{-15}
\end{align*}
so it is natural to expect that $\tau \in \theta^3_{\tnull}$. Indeed, $\tau$ is the Jacobian of the genus five curve $C$ obtained as the normalization of the singular plane octic
\begin{equation*}  
\{ x^6y^2-4x^4y^2-2x^3y^3-2x^4y+2x^3y+4x^2y^2+3xy^3+y^4+4x^2y+2xy^2+x^2-4xy-2y^2-2x+1 = 0 \}\,.
\end{equation*}
Using Macaulay2, we can write the equations of $C$ in the canonical
embedding and check that it is contained in a quadric of rank $3$, so
it has a vanishing theta null. To compute $\tau$ from $C$, we use the package \cite{BruinSage2019} in Sage \cite{SAGE}.  
\end{example}

\bibliographystyle{amsalpha}
\bibliography{bibliografia}

\end{document}